\documentclass[12pt]{amsart}
\usepackage{graphicx}
\usepackage{amsmath, amsxtra, amssymb, latexsym}
\usepackage{hyperref}

\newtheorem{thm}{Theorem}[section]

\newtheorem{lem}[thm]{Lemma}

\theoremstyle{definition}

\theoremstyle{remark}
\newtheorem{rem}[thm]{Remark}

\numberwithin{equation}{section}

\def\XXint#1#2#3{{\setbox0=\hbox{$#1{#2#3}{\int}$}
  \vcenter{\hbox{$#2#3$}}\kern-.5\wd0}}

\begin{document}

\title[Smooth metric measure spaces]
{$L_f^p$ harmonic 1-forms on complete non-compact smooth metric measure spaces}

\author{Jiuru Zhou and Peng Zhu}
\address[J.R. Zhou]{School of Mathematical Sciences, Yangzhou
University, Yangzhou, Jiangsu, China 225002.}
\email{\href{mailto: J.R.
Zhou<zhoujiuru@yzu.edu.cn>}{zhoujiuru@yzu.edu.cn}}

\address[P. Zhu]{School of Mathematics and Physics,
Jiangsu University of Technology, Changzhou,
Jiangsu, People¡¯s Republic of China 213001.}
\email{\href{mailto: P. Zhu<zhupeng2004@126.com>}{zhupeng2004@126.com}}




\date{\today}

\begin{abstract}
 This paper studies complete non-compact smooth metric measure space $(M^n,g,\mathrm{e}^{-f}\mathrm{d}v)$ with positive first spectrum $\lambda_1(\Delta_f)$ or satisfying a weighted Poincar\'e inequality with weight function $\rho$. We establish two splitting and vanishing theorems for $L_f^p$ harmonic $1$-forms under the assumption that $m$-Bakry-\'Emery Ricci curvature $\mathrm{Ric}_{m,n}\geq -a\lambda_1(\Delta_f)$ or $\mathrm{Ric}_{m,n}\geq -a\rho-b$ for particular constants $a$ and $b>0$. These results are inspired by the work of Han-Lin and are $L_f^p$ generalizations of previous works by Dung-Sung and Vieira for $L^2$ harmonic $1$-forms.
\end{abstract}


\maketitle


\section{Introduction}
For a complete non-compact Riemannian manifold $M^n$ with positive first spectrum $\lambda_1(M)$, Li and Wang \cite{LW01} have shown that the space of $L^2$ harmonic $1$-forms $\mathcal{H}^1(L^2(M))$ vanishes, provided the Ricci curvature has a negative lower bound $-a\lambda_1(M)$ for some $0<a<\dfrac{n}{n-1}$. Later, Lam \cite{La10} generalized this on manifolds satisfying a weighted Poincar\'e inequality with growth rate assumption on the weight function $\rho$. Recently, Vieira \cite{Vi16} removed the sign and growth rate assumptions on $\rho$. Interestingly, Cheng and Zhou developed Li-Wang's results in order to suitably apply to minimal hypersurfaces, see Theorem 1.2 in \cite{CZ09} and its applications. Besides, Dung-Sung's Theorem 2.2 in \cite{DS14} considered the critical case when the Ricci curvature has lower bound $-\dfrac{n}{n-1}\lambda_1(M)$, and get that either $\mathcal{H}^1(L^2(M))$ vanishes or the universal cover $\widetilde{M}$ splits as $\mathbb{R}\times N^{n-1}$; Vieira's Theorem 4 and 7 in \cite{Vi16} showed similar results provided $\mathrm{Ric}\geq -a\rho-b$ for some $0<a<\dfrac{n}{n-1}, b>0$ and the first spectrum $\lambda_1(M)\geq\dfrac{b}{\frac{n}{n-1}-a}$.

For a metric measure space $(M,g,\mathrm{e}^{-f}\mathrm{d}v)$, there are corresponding works on vanishing and splitting results related to the space of $L_f^2$ harmonic $1$-forms, see \cite{Du12,DS13,Fu14,Vi13,Wa12,Wu13,Zh20} and references therein. In particular, Han-Lin's Theorem 1.4 in \cite{HL17} considers a complete non-compact smooth metric measure space satisfying a weighted Poincar\'e inequality with a weight function $\rho$, and obtained that the space of $L_f^p$ harmonic $1$-forms $\mathcal{H}^1(L_f^p(M))$ vanishes, if the $m$-Bakry-\'Emery Ricci curvature $\mathrm{Ric}_{m,n}\geq -a\rho$ for some $a<\dfrac{4[(m-1)p-(m-2)]}{(m-1)p^2}$ with $p\geq\dfrac{m-2}{m-1}$, which can be seen as an $L_f^p$ generalization of corresponding Theorems in \cite{LW01} and \cite{Vi16}. Inspired by \cite{HL17,DS14,Vi16}, in this paper, we want to consider the geometric structure of $(M,g,\mathrm{e}^{-f}\mathrm{d}v)$ for the critical case when
$$
 \mathrm{Ric}_{m,n}\geq -\dfrac{4[(m-1)p-(m-2)]}{(m-1)p^2}\lambda_1(\Delta_f),
$$
where $\lambda_1(\Delta_f)$ is the first spectrum of the weighted Laplacian $\Delta_f$. The following is the structure of this paper:

In section 2, we recall and establish definitions, basic facts and some lemmas for metric measure spaces. Section 3 contains two main results about vanishing of space of $L_f^p$ harmonic $1$-forms and splitting for $M$. More specifically, Theorem \ref{main1} says for the critical case of assumptions on $\mathrm{Ric}_{m,n}$ with respect to Theorem 1.4 in \cite{HL17} for $\rho=\lambda_1(\Delta_f)$, either the space of $L^p_f$ harmonic $1$-forms vanishes or the universal cover $\widetilde{M}$ splits; and Theorem \ref{main2} is a vanishing and splitting result by assuming that $\mathrm{Ric}_{m,n}\geq -a\rho -b$ for some $a<\dfrac{4[(m-1)p-(m-2)]}{(m-1)p^2}, b>0$ with $p\geq\dfrac{m-2}{m-1}$ and $\lambda_1(M)\geq\dfrac{b}{\frac{4[(m-1)p-(m-2)]}{(m-1)p^2}-a}$.

\section{Preliminaries}
Recall that a smooth metric measure space $(M^n,g,\mathrm{e}^{-f}\mathrm{d}v)$ is a smooth Riemannian manifold $(M,g)$ together with a smooth function $f$ and a measure $\mathrm{e}^{-f}\mathrm{d}v$. For any constant $m\geq n$, the $m$-Bakry-\'{E}mery Ricci curvature is defined as
$$
 \mathrm{Ric}_{m,n}=\mathrm{Ric}+\nabla^2 f-\frac{\nabla f\otimes \nabla f}{m-n}.
$$
When $m=n$, we need $f$ to be a constant, then $\mathrm{Ric}_{m,n}$ is just the Ricci tensor. Hence, for the proofs in the following sections, we are actually dealing with the case when $m>n$, and $m=n$ case is similar. Analog to $L^2$ differential forms, a differential form $\omega$ is called an $L^p_f$ differential form if
$$
 \int_M |\omega|^p\cdot\mathrm{e}^{-f}\mathrm{d}v<\infty,
$$
and the space of $L_f^p$ differential $k$-forms is denoted by $\Omega^k(L_f^p(M))$.

The formal adjoint of the exterior derivative $\mathrm{d}$ with respect to the $L^2_{f}$ inner product is
$$
 \delta_f=\delta+\iota_{\nabla f}.
$$
Then the $f$-Hodge Laplacian operator is defined as
$$
 \Delta_f=-(\mathrm{d}\delta_f+\delta_f\mathrm{d}).
$$
In this paper, similar to \cite[3622]{Li09}, we define the space of $L_f^p$ harmonic $k$-forms to be
$$\mathcal{H}^k(L_f^p(M)):=\{ \alpha\in\Omega^k(L_f^p(M))~|~\mathrm{d}\alpha=0, \delta_f\alpha=0 \},$$ and define
$$\mathcal{H}_{k}(L_f^p(M)):=\ker (\Delta_f)\cap\Omega^k(L_f^p(M)).$$
By Lemma 2.2 in \cite{Vi13}, when $p=2$, we have $\mathcal{H}^k(L_f^2(M))=\mathcal{H}_{k}(L_f^2(M))$. Nevertheless, this is not the case for general $p$.
Actually, Alexandru-Rugina \cite{Al96} found $L^p~(p\not=2)$ integral $k$-forms $\alpha$ satisfying $\Delta \alpha=0$, which are neither closed nor co-closed, on hyperbolic space $\mathbb{H}^n$ for $n\geq 3$.

Recall that the first spectrum of the $f$-Laplacian $\Delta_f$ is given by
$$\lambda_1(\Delta_f)=\inf _{\psi \in C_{0}^{\infty}(M)} \frac{\int_ {M}|\nabla \psi|^{2} \cdot\mathrm{e}^{-f}\mathrm{d}v}{\int_{M} \psi^{2} \cdot\mathrm{e}^{-f}\mathrm{d}v},$$
and a natural generalization of $\lambda_1(\Delta_f)$ being positive is the weighted Poincar\'e inequality,
\begin{eqnarray}\label{wPoincare}
  \int_M \rho(x)\psi^2\cdot\mathrm{e}^{-f}\mathrm{d}v\leq\int_M |\nabla\psi|^2\cdot\mathrm{e}^{-f}\mathrm{d}v,
 \end{eqnarray}
for any test function $\psi\in C^{\infty}_0(M)$. Fu \cite{Fu14} discussed the existence of such metric measure spaces and gave explicit examples.

In the proofs of the sequel, we will always use the following cut-off function,
  \begin{eqnarray}\label{cutoff}
  \phi=\left\{
  \begin{array}{ccc}
    &1,&~~~~\mathrm{on}~B(R),\\
    &0,&~~~~\mathrm{on}~M\backslash B(2R),
  \end{array}
  \right.
 \end{eqnarray}
 such that $\displaystyle{|\nabla\phi|^2\leq \frac{C}{R^2}}$ on $B(2R) \backslash B(R)$.
\begin{lem}\label{lem1}
 Suppose $(M,g,\mathrm{e}^{-f}\mathrm{d}v)$ is a complete non-compact smooth metric measure space of dimension $n\geq 3$, and $A,B,p$ are constants satisfying $A<1, B>0, p-1+A>0$. If $h$ is a non-negative $L^p_f$ integrable function on $M$ satisfying the differential inequality
 \begin{eqnarray}\label{bochner inequ1}
  h\Delta_f h\geq A|\nabla h|^2-Bh^2
 \end{eqnarray}
 in the weak sense, then
 $$
  \int_M |\nabla h^{\frac{p}{2}}|\cdot\mathrm{e}^{-f}\mathrm{d}v\leq \frac{p^2 B}{4(p-1+A)}\int_M h^{p}\cdot\mathrm{e}^{-f}\mathrm{d}v.
 $$
\end{lem}
\begin{proof}
 Let $q=\frac{p}{2}$, and by (\ref{bochner inequ1}), we have
 \begin{eqnarray*}
  h^q\Delta_f h^q &=& h^q[q(q-1)h^{q-2}|\nabla h|^2+qh^{q-1}\Delta_f h]\\
  &\geq& \frac{q-1}{q}|\nabla h^q|^2+qAh^{2q-2}|\nabla h|^2-qBh^{2q}\\
  &=& (1-\frac{1-A}{q})|\nabla h^q|^2-qBh^{2q}.
 \end{eqnarray*}

 Multiplying both sides by the cut-off function $\phi^2$,
 $$\int_M \phi^2h^q\Delta h^q\cdot\mathrm{e}^{-f}\geq (1-\frac{1-A}{q})\int_M \phi^2|\nabla h^q|^2\cdot\mathrm{e}^{-f} -qB\int_M \phi^2 h^{2q}\cdot\mathrm{e}^{-f},$$
 where
 \begin{eqnarray*}
  &&\int_M \phi^2h^q\Delta h^q\cdot\mathrm{e}^{-f}\\
  &=& -\int_M \phi^2|\nabla h^q|^2\cdot\mathrm{e}^{-f}-2\int_M \phi h^q\langle \nabla \phi,\nabla h^q\rangle\cdot\mathrm{e}^{-f}\\
  &\leq& -\int_M \phi^2|\nabla h^q|^2\cdot\mathrm{e}^{-f}+\varepsilon \int_M \phi^2|\nabla h^q|^2\cdot\mathrm{e}^{-f}+\frac{1}{\varepsilon}\int_M h^{2q}|\nabla \phi|^2\cdot\mathrm{e}^{-f},
 \end{eqnarray*}
 for any $\varepsilon>0$.
 Hence, by the choice of $\phi$,
 $$
 \left( 2-\frac{1-A}{q}-\varepsilon \right)\int_M \phi^2|\nabla h^q|^2\cdot\mathrm{e}^{-f}\leq qB\int_M \phi^2 h^{2q}\cdot\mathrm{e}^{-f}+\frac{C}{\varepsilon R^2}\int_M h^{2q}\cdot\mathrm{e}^{-f}.
 $$
 Letting $R\rightarrow\infty$, it holds
 $$
 \left( 2-\frac{1-A}{q}-\varepsilon \right)\int_M |\nabla h^q|^2\cdot\mathrm{e}^{-f}\leq qB\int_M  h^{2q}\cdot\mathrm{e}^{-f},
 $$
 for any $\varepsilon>0$, so
 $$
 \left( 2-\frac{1-A}{q} \right)\int_M |\nabla h^q|^2\cdot\mathrm{e}^{-f}\leq qB\int_M  h^{2q}\cdot\mathrm{e}^{-f}.
 $$
 \end{proof}

The following is an $L_f^p$ generalization of Lemma 2 and 3 in \cite{Vi16},
\begin{lem}\label{lem2}
 Suppose $(M,g,\mathrm{e}^{-f}\mathrm{d}v)$ is a complete non-compact smooth metric measure space of dimension $n\geq 3$ satisfying a weighted Poincar\'e inequality with weight function $\rho$, and $A,a,b,p$ are constants satisfying $A<1, a>0, b>0, 4(p-1+A)-p^2 a>0$. If $h$ is a non-negative $L^p_f$ integrable function on $M$ satisfying the differential inequality
 \begin{eqnarray}\label{bochner inequ2}
  h\Delta_f h\geq A|\nabla h|^2-a\rho h^2-bh^2
 \end{eqnarray}
 in the weak sense, then
 \begin{eqnarray}\label{dirichlet2}
  \int_M |\nabla h^{\frac{p}{2}}|\cdot\mathrm{e}^{-f}\mathrm{d}v\leq \frac{p^2 b}{4(p-1+A)-p^2a}\int_M h^{p}\cdot\mathrm{e}^{-f}\mathrm{d}v.
 \end{eqnarray}
 Moreover,
 \begin{eqnarray}\label{eigen inequ}
  \lambda_1(\Delta_f)\int_M  h^p\cdot\mathrm{e}^{-f}\mathrm{d}v\leq \int_M |\nabla h^{\frac{p}{2}}|^2\cdot\mathrm{e}^{-f}\mathrm{d}v.
 \end{eqnarray}
\end{lem}

\begin{proof}
 Similar to the proof of Lemma \ref{lem1},
 \begin{eqnarray*}
  h^q\Delta_f h^q
  \geq \left(1-\frac{1-A}{q}\right)|\nabla h^q|^2-qa\rho h^{2q}-qbh^{2q},
 \end{eqnarray*}
 where $2q=p$.
 Multiplying both sides of by the cut-off function $\phi^2$ and integrating by parts, it holds
 \begin{eqnarray*}
  &&
  \left(1-\frac{1-A}{q}\right)\int_M \phi^2|\nabla h^q|^2\cdot\mathrm{e}^{-f} -qa\int_M \rho(\phi h^{q})^2\cdot\mathrm{e}^{-f}-qb\int_M \phi^2h^{2q}\cdot\mathrm{e}^{-f}\\
  &\leq&
   -\int_M \phi^2|\nabla h^q|^2\cdot\mathrm{e}^{-f}-2\int_M \phi h^q\langle \nabla \phi,\nabla h^q\rangle\cdot\mathrm{e}^{-f}
 \end{eqnarray*}
 Hence, combining the weighted Poincar\'e inequality, we obtain
 \begin{eqnarray*}
  &&
  \left(2-\frac{1-A}{q}\right)\int_M \phi^2|\nabla h^q|^2\cdot\mathrm{e}^{-f} \\
  &\leq&
   qa\int_M \rho(\phi h^{q})^2\cdot\mathrm{e}^{-f}-2\int_M \phi h^q\langle \nabla \phi,\nabla h^q\rangle\cdot\mathrm{e}^{-f}+qb\int_M \phi^2h^{2q}\cdot\mathrm{e}^{-f}\\
   &\leq&
   qa\int_M |\nabla(\phi h^{q})|^2\cdot\mathrm{e}^{-f}-2\int_M \phi h^q\langle \nabla \phi,\nabla h^q\rangle\cdot\mathrm{e}^{-f}+qb\int_M \phi^2h^{2q}\cdot\mathrm{e}^{-f}\\
   &\leq&
   (qa+|qa-1|\varepsilon)\int_M \phi^{2}|\nabla h^q |^2\cdot\mathrm{e}^{-f}\\
   &&+\left(\frac{|qa-1|}{\varepsilon}+qa\right)\int_M h^{2q}|\nabla\phi |^2\cdot\mathrm{e}^{-f}+qb\int_M \phi^2h^{2q}\cdot\mathrm{e}^{-f},
 \end{eqnarray*}
 then, by the choice of $\phi$,
 \begin{eqnarray*}
  &&
  \left(2-\frac{1-A}{q}-qa-|qa-1|\varepsilon\right)\int_M \phi^2|\nabla h^q|^2\cdot\mathrm{e}^{-f} \\
  &\leq&
  \left(\frac{|qa-1|}{\varepsilon}+qa\right)\frac{C}{R^2}\int_M h^{2q}\cdot\mathrm{e}^{-f}+qb\int_M \phi^2h^{2q}\cdot\mathrm{e}^{-f}.
 \end{eqnarray*}
 Letting $R\rightarrow \infty$, and then $\varepsilon\rightarrow 0$, we get (\ref{dirichlet2}).

 By variational principle and Cauchy-Schwartz inequality,
 \begin{eqnarray*}
  \lambda_1(\Delta_f)\int_M \phi^{2}h^{2q}\cdot\mathrm{e}^{-f}&\leq& \int_M |\nabla (\phi h^q) |^2\cdot\mathrm{e}^{-f}\\
  &\leq& \int_M \phi^{2}|\nabla h^q |^2\cdot\mathrm{e}^{-f}+\int_M h^{2q}|\nabla \phi |^2\cdot\mathrm{e}^{-f}\\
  &&+2\left( \int_M \phi^{2}|\nabla h^q |^2\cdot\mathrm{e}^{-f} \right)^{\frac{1}{2}}\left( \int_M h^{2q}|\nabla \phi |^2\cdot\mathrm{e}^{-f} \right)^{\frac{1}{2}}.
 \end{eqnarray*}
 Letting $R\rightarrow\infty$, we obtain (\ref{eigen inequ}).
\end{proof}

Notice that an $L^p$ harmonic form $\alpha$ satisfies $\mathrm{d}\alpha=0$ and $\delta_f\alpha=0$, the following Bochner type inequality for $L_f^2$ harmonic 1-forms in \cite{Vi13,Zh20} also holds, which is compared to Lemma 2.1 in \cite{Li05} for $f$-harmonic functions.
\begin{lem}\label{bochner inequ}
 Let $\omega$ be an $L_f^p$ harmonic 1-form on an $n$-dimensional complete smooth metric measure space $(M,g,\mathrm{e}^{-f}\mathrm{d}v)$ and $m\geq n$ be any constant. Then
 \begin{eqnarray}\label{bochner inequality}
  |\omega|\Delta_f|\omega|\geq\frac{|\nabla|\omega||^2}{m-1}+\mathrm{Ric}_{m,n}(\omega,\omega).
 \end{eqnarray}
 Equality holds iff
 $$
  \left(\omega_{i,j}\right)=\left(\begin{array}{ccccc}{-(m-1) \mu} & {0} & {0} & {\ldots} & {0} \\
  {0} & {\mu} & {0} & {\ldots} & {0} \\ {0} & {0} & {\mu} & {\ldots} & {0} \\
  {\vdots} &  {\vdots} & {\vdots} & {\ddots} & {\vdots} \\
  {0} &  {0} & {0} & {\ldots} & {\mu}\end{array}\right),
 $$
 where $\mu=\dfrac{\langle\nabla f,\omega\rangle}{n-m}$.
\end{lem}

\section{Vanishing and splitting results}

We present the main results in this section, and adopt Han-Lin and Vieira's idea in \cite{HL17}, \cite{Vi16} for the proofs.
\begin{thm}\label{main1}
 Suppose $(M,g,\mathrm{e}^{-f}\mathrm{d}v)$ is a complete non-compact smooth metric measure space of dimension $n\geq 3$ with positive first spectrum $\lambda_1(\Delta_f)$, and $p>\dfrac{m-2}{m-1}$ is a constant. If the $m$-Bakry-\'Emery Ricci curvature satisfies
 \begin{eqnarray*}
   \mathrm{Ric}_{m,n} \geq -\dfrac{4[(m-1)p-(m-2)]}{(m-1)p^2}\lambda_1(\Delta_f),
 \end{eqnarray*}
 then, either\\
 $(1).$ $\mathcal{H}^1(L_f^p(M))=\{0\}$; or\\
 $(2).$ $\widetilde{M}=\mathbb{R}\times N$, where $\widetilde{M}$ is the universal cover of $M$ and $N$ is a manifold of dimension $n-1$.
\end{thm}

\begin{proof}
 If $\mathcal{H}^1(L_f^p(M))=\{0\}$, it's done. Otherwise, choose a non-trivial $L_f^{p}$ harmonic $1$-form $\omega\in\mathcal{H}^1(L_f^p(M))$, and let $h=|\omega|$, $2q=p$, $a=\dfrac{4[(m-1)p-(m-2)]}{(m-1)p^2}$. By (\ref{bochner inequality}) and the condition on $\mathrm{Ric}_{m,n}$, we have
 \begin{eqnarray}\label{bochner inequ3}
  h\Delta_f h\geq \frac{1}{m-1}|\nabla h|^2-a\lambda_1(\Delta_f)h^2.
 \end{eqnarray}
 Applying Lemma \ref{lem1} with $A=\dfrac{1}{m-1}$ and $B=a\lambda_1(\Delta_f)$, we obtain
 $$
  \int_M |\nabla h^{\frac{p}{2}}|\cdot\mathrm{e}^{-f}\mathrm{d}v\leq \lambda_1(\Delta_f)\int_M h^{p}\cdot\mathrm{e}^{-f}\mathrm{d}v,
 $$
 so the integral $\int_M |\nabla h^{\frac{p}{2}}|\cdot\mathrm{e}^{-f}\mathrm{d}v$ is finite.
 Moreover, (\ref{bochner inequ3}) implies
 \begin{eqnarray}\label{pbochner inequ1}
  0\leq h^q\Delta_f h^q-\frac{q(m-1)-(m-2)}{q(m-1)}|\nabla h^q|^2+qa\lambda_1(\Delta_f)h^{2q}.
 \end{eqnarray}
 Then by multiplying the cut-off function $\phi^2$ and taking integration, we get
 \begin{eqnarray*}
  0&\leq& \int_M \left( h^q\Delta_f h^q-\frac{q(m-1)-(m-2)}{q(m-1)}|\nabla h^q|^2+qa\lambda_1(\Delta_f)h^{2q} \right)\phi^2\cdot \mathrm{e}^{-f}\\
  &\leq& -\int_M \langle \nabla(\phi^2 h^q),\nabla h^q \rangle\cdot \mathrm{e}^{-f}-\frac{q(m-1)-(m-2)}{q(m-1)}\int_M \phi^2 |\nabla h^q|^2 \cdot \mathrm{e}^{-f}\\
  && + qa\int_M |\nabla(\phi h^q)|^2\cdot \mathrm{e}^{-f}\\
  &\stackrel{R\rightarrow\infty}{\longrightarrow}& -\int_M |\nabla h^q|^2 \cdot \mathrm{e}^{-f}-\frac{q(m-1)-(m-2)}{q(m-1)}\int_M |\nabla h^q|^2 \cdot \mathrm{e}^{-f}\\
  && + qa\int_M |\nabla h^q|^2\cdot \mathrm{e}^{-f}\\
  &=& 0.
 \end{eqnarray*}
 This forces equality holds in (\ref{pbochner inequ1}), and so is that in (\ref{bochner inequ3}). Hence, the equality holds in (\ref{bochner inequality}). By Lemma \ref{bochner inequ},
 $$
  \left(\omega_{i,j}\right)=\left(\begin{array}{ccccc}{-(m-1) \mu} & {0} & {0} & {\ldots} & {0} \\
  {0} & {\mu} & {0} & {\ldots} & {0} \\ {0} & {0} & {\mu} & {\ldots} & {0} \\
  {\vdots} &  {\vdots} & {\vdots} & {\ddots} & {\vdots} \\
  {0} &  {0} & {0} & {\ldots} & {\mu}\end{array}\right),
 $$
 where $\mu=\dfrac{\langle\nabla f,\omega\rangle}{n-m}$. Suppose $\widetilde{M}$ is a universal cover of $M$, and $\widetilde{\omega}$ is the lifting of $\omega$ on $\widetilde{M}$. Since $\widetilde{M}$ is simply connected and $\widetilde{\omega}$ is closed, there exists a function $u$ on $\widetilde{M}$ such that $\mathrm{d}u=\widetilde{\omega}$. Hence, the Hessian matrix of $u$ is given by
 $$
  \left(u_{,ij}\right)=\left(\begin{array}{ccccc}{-(m-1) \nu} & {0} & {0} & {\ldots} & {0} \\
  {0} & {\nu} & {0} & {\ldots} & {0} \\ {0} & {0} & {\nu} & {\ldots} & {0} \\
  {\vdots} &  {\vdots} & {\vdots} & {\ddots} & {\vdots} \\
  {0} &  {0} & {0} & {\ldots} & {\nu}\end{array}\right),
 $$
 where $\nu=\dfrac{\langle\nabla f,\nabla u\rangle}{n-m}$.
 The splitting argument is the similar to that of Li and Wang \cite{LW06}, page 946, and the same as Fu \cite{Fu14}, page 10, so we omit here.
\end{proof}

\begin{rem}
 If $p=2, m=n$, this is just Theorem 2.2 in \cite{DS14}. Besides, if $p=\dfrac{m-2}{m-1}$, then $\dfrac{4[(m-1)p-(m-2)]}{(m-1)p^2}=0$, and the curvature assumption becomes $\mathrm{Ric}_{m,n} \geq 0$. This is the critical case of the conditions in Theorem 3.2 in \cite{HL17}, and it's interesting to consider the geometry of such metric measure spaces.
\end{rem}

Like Vieira did in Theorem 4 and 7 in \cite{Vi16} for $L^2$ harmonic $1$-forms, here we show that with a more general curvature assumption, if the first spectrum of the $f$-Laplacian has certain lower bound, we have
\begin{thm}\label{main2}
 Suppose $(M,g,\mathrm{e}^{-f}\mathrm{d}v)$ is a complete non-compact smooth metric measure space of dimension $n \geq 3$ satisfying a weighted Poincar\'e inequality with weight function $\rho$. Assume the $m$-Bakry-\'Emery Ricci curvature satisfies
  $$\mathrm{Ric}_{m,n}\geq -a\rho-b,$$
  for some $a<\dfrac{4[(m-1)p-(m-2)]}{(m-1)p^2}$ with $p\geq \dfrac{m-2}{m-1}$ and $b>0$.

  If $\lambda_1(\Delta_f)>\dfrac{b}{\frac{4[(m-1)p-(m-2)]}{(m-1)p^2}-a}$, then $H^1(L^p_f(M))=\{0\}$.

  Moreover, if $\lambda_1(\Delta_f)= \dfrac{b}{\frac{4[(m-1)p-(m-2)]}{(m-1)p^2}-a}$,
 then either

 $(1)$. $H^1(L^p_f(M))=\{0\}$; or

 $(2)$. $\widetilde{M}=\mathbb{R}\times N$, where $\widetilde{M}$ is the universal cover of $M$ and $N$ is a manifold of dimension $n-1$.
\end{thm}
\begin{proof}
 For any $L_f^p$ harmonic 1-form $\omega$, let $h=|\omega|$, so
 \begin{eqnarray}\label{bochner func2}
  h\Delta_f h\geq \frac{1}{m-1}|\nabla h|^2-a\rho h^2-bh^2.
 \end{eqnarray}
 Applying Lemma \ref{lem2} with $A=\dfrac{1}{m-1}$, we have
 \begin{eqnarray}\label{int inequal}
  \int_M |\nabla h^{\frac{p}{2}}|^2\cdot \mathrm{e}^{-f}\mathrm{d}v
  \leq \dfrac{b}{\frac{4[(m-1)p-(m-2)]}{(m-1)p^2}-a}\int_M h^p\cdot \mathrm{e}^{-f}\mathrm{d}v<\infty,
 \end{eqnarray}
 and by (\ref{eigen inequ}), we have
 \begin{eqnarray}\label{Pinequ}
  \lambda_{1}(\Delta_f)\int_M h^p\cdot \mathrm{e}^{-f}\mathrm{d}v
  \leq \int_M |\nabla h^{\frac{p}{2}}|^2\cdot \mathrm{e}^{-f}\mathrm{d}v.
 \end{eqnarray}
 Suppose $\lambda_1(\Delta_f)>\dfrac{b}{\frac{4[(m-1)p-(m-2)]}{(m-1)p^2}-a}$, if $\omega$ is non-trivial, then (\ref{int inequal}) and (\ref{Pinequ}) implies
 $$\lambda_1(\Delta_f)\leq\dfrac{b}{\frac{4[(m-1)p-(m-2)]}{(m-1)p^2}-a},$$
 which is a contradiction.
 Hence, in this case, $\mathcal{H}^1(L_f^p(M))=\{ 0 \}$.

 Suppose $\lambda_1(\Delta_f)=\dfrac{b}{\frac{4[(m-1)p-(m-2)]}{(m-1)p^2}-a}$ and $\mathcal{H}^1(L_f^p(M))$ is non-trivial, then (\ref{int inequal}) and (\ref{Pinequ}) imply that equality holds in (\ref{int inequal}). In addition, by (\ref{bochner func2}),
 \begin{eqnarray}\label{bochner func3}
  0\leq h^{\frac{p}{2}}\Delta_f h^{\frac{p}{2}} - \left[1-\frac{2(m-2)}{p(m-1)}\right]|\nabla h^{\frac{p}{2}}|^2+\frac{p}{2}a\rho h^p+\frac{p}{2}bh^p.
 \end{eqnarray}
Hence, by multiplying $\phi^2$ on both sides of (\ref{bochner func3}), taking integration, integrating by parts and the weighted Poincar\'e inequality, we have
 \begin{eqnarray*}
  0 &\leq& \int_M \left(h^{\frac{p}{2}}\Delta_f h^{\frac{p}{2}} - \left(1-\frac{2(m-2)}{p(m-1)}\right)|\nabla h^{\frac{p}{2}}|^2+\frac{p}{2}a\rho h^p+\frac{p}{2}bh^p \right)\phi^2\cdot\mathrm{e}^{-f}\\
  &\leq& -\int_M \langle \nabla(\phi^2 h^{\frac{p}{2}}),\nabla h^{\frac{p}{2}} \rangle \cdot\mathrm{e}^{-f}-\left[1-\frac{2(m-2)}{p(m-1)}\right]\int_M |\nabla h^{\frac{p}{2}}|^2\phi^2\cdot\mathrm{e}^{-f}\\
  && +\frac{ap}{2}\int_M |\nabla (\phi h^{\frac{p}{2}})|^2\cdot\mathrm{e}^{-f}+\frac{bp}{2}\int_M \phi^2 h^p\cdot\mathrm{e}^{-f}\\
  &\stackrel{R\rightarrow\infty}{\longrightarrow}& -\int_M |\nabla h^{\frac{p}{2}}|^2 \cdot\mathrm{e}^{-f}-\left[1-\frac{2(m-2)}{p(m-1)}\right]\int_M |\nabla h^{\frac{p}{2}}|^2\cdot\mathrm{e}^{-f}\\
  && +\frac{ap}{2}\int_M |\nabla h^{\frac{p}{2}}|^2\cdot\mathrm{e}^{-f}+\frac{bp}{2}\int_M  h^p\cdot\mathrm{e}^{-f}\\
  &=& \frac{p}{2}\left[\left( -\frac{4((m-1)p-(m-2))}{p^2(m-1)}+a \right)\int_M |\nabla h^{\frac{p}{2}}|^2 \cdot\mathrm{e}^{-f}+b\int_M  h^p\cdot\mathrm{e}^{-f}\right]\\
  &=& 0,
 \end{eqnarray*}
 since equality holds in (\ref{int inequal}).
 This forces equality holds in (\ref{bochner func3}), and so is equality in (\ref{bochner func2}). Finally, equality holds in (\ref{bochner inequality}). Then the splitting argument is the same as that in Theorem \ref{main1}.
 \end{proof}
\begin{rem}
 If $p=2, m=n$, this is just Theorem 4 and 7 in \cite{Vi16}.
\end{rem}

\vspace{0.5cm}\noindent\textbf{Acknowledgments}.  The authors would like to thank Prof. Xingwang Xu and Prof. Detang Zhou for continual support of their research. The first author is partially supported by NSFC Grants 11771377 and the Natural Science Foundation of Jiangsu Province (BK20191435). The second author is partially supported by NSFC Grants 11801229.

\end{document}